\documentclass[11pt]{article}
\usepackage{amssymb, latexsym, mathtools, amsthm, amsmath}
\begingroup
    \makeatletter
    \@for\theoremstyle:=definition,remark,plain\do{%
        \expandafter\g@addto@macro\csname th@\theoremstyle\endcsname{%
            \addtolength\thm@preskip\parskip
            }%
        }
\endgroup
\usepackage{url}
\usepackage[margin=1in]{geometry}
 \renewenvironment{abstract}
 { \normalsize
  \list{}{\setlength{\leftmargin}{.0cm}%
    \setlength{\rightmargin}{\leftmargin}}%
  \item {\bf \abstractname.} \relax}
 {\endlist}
\usepackage{xcolor}

\definecolor{dnrbl}{rgb}{0,0,0.3}
\definecolor{dnrgr}{rgb}{0,0.3,0}
\definecolor{dnrre}{rgb}{0.5,0,0}
\usepackage[colorlinks=true, citecolor=dnrgr, linkcolor=dnrre, urlcolor=dnrre] {hyperref}

\usepackage{enumerate}
\usepackage[round]{natbib}

\theoremstyle{plain}
\newtheorem{thm}{Theorem}[section]

\newtheorem{lem}[thm]{Lemma}

\theoremstyle{definition}

\newtheorem{defi}[thm]{Definition}
\newcommand{\Nat}{\mathbb{N}}

\newcommand{\restr}{\upharpoonright}  %restriction

\newcommand{\bigo}[1]{\mathop{\bf O}\/\left({#1}\right)}

\newcommand{\sqbrad}[2]{\left\{\ {#1} : {#2}\ \right\}}

\newcommand{\bigon}{\textrm{\textup O}\hspace{0.02cm}(1)}

\newcommand{\MM}{\mathcal{M}}

\newcommand{\abs}[1]{\left|{#1}\right|}

\newcommand{\parb}[1]{\big({#1}\big)}

\newcommand{\Dim}{\mathrm{Dim}}

\newcommand{\ml}{Martin-L\"{o}f }

\newcommand{\ce}{c.e.\ }

\newcommand{\lce}{left-c.e.\ }

\newcommand{\pf}{prefix-free }

\newcommand{\fand}{\ \wedge\ }

\newcommand{\twomel}{2^{<\omega}}

\title{Aspects of Muchnik's paradox in restricted betting 
\thanks{Supported by NSFC grant No.~11971501.}}
\author{George Barmpalias  \and Lu Liu} 
\date{\today}
\begin{document}
\maketitle
\begin{abstract}
Muchnik's paradox says that enumerable betting strategies are not always reducible to
enumerable strategies whose bets are restricted to either even rounds or odd rounds. 
In other words, there are outcome sequences $x$ where an effectively enumerable strategy succeeds, but no
such parity-restricted effectively enumerable strategy does. We show that any such $x$ has effective Hausdorff dimension
at least $1/2$. We also construct reals that are random with respect to parity-restricted effectively enumerable strategies
and have effective packing dimension at most $\log\sqrt{3}$.
Finally we exhibit Muchnik's paradox in the case of computable integer-valued strategies.
\end{abstract}
\vspace*{\fill}
\noindent{\bf George Barmpalias}\\[0.5em]
\noindent State Key Lab of Computer Science,
Institute of Software, Chinese Academy of Sciences, Beijing, China. 
\textit{E-mail:} \texttt{\textcolor{dnrgr}{barmpalias@gmail.com}}.
\textit{Web:} \texttt{\textcolor{dnrre}{http://barmpalias.net}}\par
\addvspace{\medskipamount}\medskip\medskip\medskip\medskip
\noindent{\bf Lu Liu}\\[0.2em]  
\noindent Department of Mathematics, Central South University,
ChangSha 410083, China \\%[0.2em]
\textit{E-mail:} \texttt{\textcolor{dnrgr}{g.jiayi.liu@gmail.com.}}
%\textit{Web:} \texttt{\textcolor{dnrre}{http://lewis-pye.com.}} 
\vfill \thispagestyle{empty}
\clearpage
\setcounter{tocdepth}{1}
\tableofcontents

\section{Introduction}
A {\em casino} may be determined by the set of binary outcome sequences $x$ that it produces, where
sequential bets can be placed by players using {\em wagers} they can afford (namely at most their current capital).
A betting strategy may be qualified by the amount of information that it uses, for example its computational or arithmetical complexity.
Suppose that a player can amass unbounded fortune while betting on a casino sequence $x$, starting from a finite {\em initial capital}.
\begin{equation}\label{ZilqXlk8ry}
\parbox{12cm}{Can this player also win when she is restricted to bet on a 
specific subset of stages/bits, say the even or the odd bits of $x$?}
\end{equation}
It turns out that the answer depends on the computational complexity or effectivity of the strategies considered.
To be more precise, we introduce some standard formalism.
Following \citet{villemartin}, a betting strategy can be modeled by a non-negative function $\sigma\mapsto M(\sigma)$ 
on the finite binary strings $\twomel$ representing the capital of the player at state $\sigma$, such that 
\begin{equation}\label{rc2wGxhGY9}
M(\sigma)=\frac{M(\sigma\ast 0)+M(\sigma\ast 1)}{2}.
\end{equation}
We call such functions {\em martingales} and note that they have a probabilistic interpretation: the expected {\em capital gain} 
after a bet is zero. In this formalization of strategies the  wagers and favorable outcomes are implicit: 
if $M(\sigma\ast 0)<M(\sigma\ast 1)$ then the player bets on 1 the amount $M(\sigma\ast 1)-M(\sigma)$,
which he doubles or loses depending on the outcome. The case
$M(\sigma\ast 0)>M(\sigma\ast 1)$ is analogous while $M(\sigma\ast 0)=M(\sigma\ast 1)$ means that the player does not bet. 
Betting strategies that account for inflation or saving (permanently taking capital out of the casino) 
are modeled by {\em supermartingales} which are functions $M$ such that
\eqref{rc2wGxhGY9} holds with $=$ replaced by $\geq$, hence the expected gain is no more than 0.
We say that {\em $M$ succeeds on $x$} if $\limsup_n M(x\restr_n)=\infty$ and say that $M$ is 
\begin{itemize}
\item {\em even-parity} if $M(\sigma\ast i)=M(\sigma)$ for all $\sigma$ of even length and $i<2$;
\item {\em odd-parity} if $M(\sigma\ast i)=M(\sigma)$ for all $\sigma$ of odd length and $i<2$;
\end{itemize}
and refer to either case as a {\em single-parity (super)martingale}.

\citet{muchnik2009algorithmic} showed  for the class of computable strategies, question \eqref{ZilqXlk8ry} has a positive answer:
\begin{equation}\label{REDed5hDo}
\parbox{13cm}{every computable martingale is the product
of an even-parity and an odd-parity martingale, and the same is true of supermartingales.}
\end{equation}
Betting on $x$ with  an even-parity strategy is quite stronger to 
betting on a casino sequence consisting of the even bits of $x$: an even betting strategy has access to all the previous 
(including the odd) bits of $x$, which may inform the current bet.  For example, if $x$ is algorithmically random then
a computable even-parity strategy succeeds on $x\oplus x$ but no computable strategy succeeds 
on $x$.\footnote{$x\oplus y:=x_0 y_0 x_1 y_1\dots$, where 
$x:=x_0 x_1\dots$ and  $y:=y_0 y_1\dots$.}
Computable martingales were proposed as a foundation of algorithmic information by \citet{Schnorr75}
but do not capture the standard notion of randomness by \citet{MR0223179}. For the latter we need to consider the 
class of {\em \lce (super)martingales} $M$, where a function is \lce if its values are 
uniformly computably approximable from below: $x$ is \ml random iff no \lce (super)martingale succeeds on $x$. 
A \lce (super)martingale whose wager-function is also \lce is called {\em strongly \lce (super)martingale};
note that a \lce (super)martingale is not necessarily strongly \lce

\citet{muchnik2009algorithmic} showed that \eqref{ZilqXlk8ry} has a negative answer for the class of \lce (super)martingales:
\begin{equation}\label{kHxsCRcey4}
\parbox{13cm}{{\em Muchnik's paradox:} there exists $x$ such that some \lce martingale succeeds on $x$ but no
even-parity or odd-parity \lce (super)martingale succeeds on $x$.}
\end{equation}
This means that, in the context of enumerable bets (\lce strategies) a player can win on $x$, 
but can no longer win if restricted to single-parity strategies.
This  has been called Muchnik's paradox and was discussed extensively
by \citet{ChernovSVV08} in the context of {\em online complexity} and prediction, as well as by 
\citet{stacsBauwens14} who produced Kolmogorov complexity versions of it and
related it to real-life scenarios and the theory of causality of timeseries.

\citet{muchnik2009algorithmic} does not give any information about the descriptive complexity of $x$ in 
\eqref{kHxsCRcey4}, other than it is not \ml random.
The present work deals with two aspects of \eqref{kHxsCRcey4}:
\begin{enumerate}[\hspace{0.3cm}(a)]
\item  How predictable can $x$ be, while still immune to single-parity guessing? 
\item Does this paradox occur in the context of other strategy restrictions such as minimum bets?
\end{enumerate}
{\bf Our results.} In \S\ref{TCPLGlg6St} we show that there exists $x$ which is strongly predictable, in the sense that
\begin{equation}\label{OP8RKSNpj}
\parbox{13cm}{betting on successive {\em pairs  of bits} we need at 
most three guesses to guess correctly}
\end{equation}
(out of the possible four possibilities) 
 but no single-parity \lce supermartingale succeeds on $x$. This strong version of \eqref{kHxsCRcey4}
implies the following.
\begin{thm}[Packing dimension and single-parity betting]\label{Px6x4FnWQw}
There exists $x$ with effective packing dimension 
$\Dim (x)\leq \log \sqrt{3}$ and all single-parity  \lce supermartingales fail on $x$. 
\end{thm}

In particular, the effective Hausdorff dimension $\dim (x)$ of this $x$ is at most $\log \sqrt{3}\approx 0.79$.
As we explain in \S\ref{EznXVtJqhv}, this means that some \lce martingale succeeds on $x$ {\em exponentially fast}.

We also obtain the following lower bound on the effective Hausdorff dimension of reals that can exhibit
Muchnik's paradox.
\begin{thm}[Hausdorff dimension and single-parity betting]\label{OYNIGDVIET}
If $\dim(x)<1/2$ then there exist even-parity and odd-parity (strongly)
\lce martingales $N,T$ which succeed on $x$:
$\lim_n N(x\restr_n)=\lim_n T(x\restr_n)=\infty$.
\end{thm}

In \S\ref{9tswL9nmI} we examine the above problem, namely the reducibility of a strategy to single-parity strategies,
in the context of computable {\em integer-valued (super)martingales}.
Such restricted strategies have been studied by
\citet{BienvenuST10}, \citet{Teu14agCWGnp},  \citet{JCSSbarmp15}
and are a special case of the the restricted wager strategies
studied by \citet{TeutChalcraft}, \citet{Peretzagainst}, \citet{Peretzwager}
 and \citet{granuma}. This aspect is interesting for the following reason:
 on the one hand computable strategies are reducible to single-parity strategies
 by \eqref{REDed5hDo}, on the other hand this decomposition by 
\citet{muchnik2009algorithmic} involves {\em scaling wagers} which is not possible if the strategies are
integer-valued.

\begin{thm}[Integer-valued]\label{5C9MB4KgaYa}\ 
There exists $z$ such that
some computable integer-valued martingale succeeds on $z$ but 
each computable integer-valued single-parity supermartingale fails on $z$.
Moreover we can ensure that $\dim(z)=0$.
%There exists $z$ such that
%some computable integer-valued martingale succeeds on $z$ but 
%each computable integer-valued single-sided supermartingale fails on $z$.
\end{thm}
Single-parity strategies have an apparent similarity  to the {\em monotonous} or
{\em single-sided} strategies studied by \citet{barmpalias2020monotonous}: these are strategies that
bet on a single outcome, 0 or 1. Indeed the analogue of the decomposition \eqref{REDed5hDo}
as well as Theorem \ref{OYNIGDVIET} were shown for single-sided strategies, and in \S\ref{9tswL9nmI}
we prove an analogue of Theorem \ref{5C9MB4KgaYa} for single-sided integer-valued strategies.
However \lce single-sided strategies can be much more complex than single-parity strategies as we explain in
\S\ref{SIqhiJbskS}. For the special case of {\em strongly} \lce single-sided strategies,
\citet{barmpalias2020monotonous} constructed a real of effective Hausdorff dimension $1/2$
on which all such martingales fail.
Recently, using different methods, 
\citet{sided21} showed that there exists a non-random $x$ such that all single-sided \lce supermartingales fail on $x$.
It is not known whether such a real can have effective Hausdorff dimension $<1$. This and related problems are
discussed in \S\ref{mcSrWpox} along with a summary of our results.

{\bf A word on notation.}
We maintain some regularity in our notation, using variables
$i, j, s$ for non-negative integers, $\sigma,\tau$
for strings, $x,y,z$ for reals, $M,N,T, D$ for (super)martingales. 

\section{Background on enumerable and restricted strategies}\label{EznXVtJqhv}
We show some facts about (super)martingales that are used in the next sections and 
background that is directly relevant to our results and is not readily available in the literature in this form.
General background can be found in the textbooks \citep{shenbook} and \citep{rodenisbook}.

\subsection{Enumerable strategies and restrictions}\label{SIqhiJbskS}
Intuitively, the complexity of a strategy is an expression of the information that the casino has about the player:
if a casino plays against a class of highly complex strategies, the outcome sequences that they need to produce in order to
ensure that the strategies do not win is also complex. We typically gauge complexity in terms of arithmetical definability or
computability. We say that $x$ is {\em random with respect to a class $\MM$} of (super)martingales if no member of $\MM$ succeeds on $x$.

Enumerable strategies represent a step of complexity beyond the computable strategies.
From the point of view of a constructive observer, bets are placed on the various outcomes at various times (stages),
and the final capital at a particular {\em state} (namely prefix of binary outcomes) of the game is only a limit of a computable
sequence of compound bets. This is made precise by:
\begin{equation}\label{6hgXrGpsnt}
\parbox{13cm}{A martingale is \lce if and only if it can be written as the sum of a uniformly computable
sequence of  martingales. \citep[Lemma 2.5]{barmpalias2020monotonous}.}
\end{equation}
Every computable supermartingale is bounded above by a computable martingale, but it is well-known that
the optimal \lce supermartingale is not bounded above by any \lce martingale. In fact, there is no uniform
list of all \lce martingales, and this is the reason it is often easier to work with supermartingales; this can sometimes
be done without loss of generality as whenever some \lce supermartingale succeeds on a real $x$, there exists a 
\lce martingale that succeeds on $x$. Another fact concerns the standard winning condition
$\limsup_n M(x\restr_n)=\infty$  to the stronger $\lim_n M(x\restr_n)=\infty$. It is known that
if $\limsup_n M(x\restr_n)=\infty$ for a \lce supermartingale $M$ then there exists a \lce martingale $N$ such that
$\lim_n N(x\restr_n)=\infty$; see \citep[\S 6.3]{rodenisbook}.

In the case of restricted strategies one needs to be careful as many of the above facts may fail.
For example, for some wager-restricted martingales, including the integer-valued martingales, 
the equivalence between the two winning conditions fails; see \citep{Teu14agCWGnp, granuma}.

Single-parity strategies preserve more of the standard properties above due to the fact that, by
\citet{ChernovSVV08}, they can be expressed as unrestricted strategy with {\em online access} to an oracle
(the previous outcomes that were off-limits for betting). We may say that $(\sigma,\tau)\mapsto N(\tau\mid \sigma)$ 
is an {\em online martingale} if for each $|\sigma|=|\tau|$:
\[
N(\hat\tau\ 0\mid \sigma)+N(\hat\tau\ 1\mid \sigma)=2N(\hat\tau\mid \hat\sigma)
\]
where $\hat\sigma,\hat\tau$ are the predecessors of $\sigma,\tau$ (similar for supermartingales).
This means that the bet placed at $\hat\tau$ is conditional on the last bit of $\sigma$.
Then it is clear that if $M$ is a \lce even-parity martingale, there exists an online \lce martingale $N$ 
such that $M(\sigma\oplus \tau)=N(\tau\mid \sigma)$.
\begin{lem}\label{ttrAuakcfr}
If $M$ is an even-parity \lce martingale then
\begin{enumerate}[\hspace{0.0cm}(a)]
\item there exists a uniformly computable family $(N_i)$ of even-parity martingales such that $M=\sum_i N_i$.
\item if $\limsup_n M(x\restr_n)=\infty$ there exists a \lce even-parity martingale $N$ with $\lim_n N(x\restr_n)=\infty$.
\end{enumerate}
Both clauses also hold for odd-parity martingales, and (b) also holds when $M$ is a supermartingale.
\end{lem}\begin{proof}
A direct relativization of \citep[Lemma 2.5]{barmpalias2020monotonous} shows that
\eqref{6hgXrGpsnt} holds for online martingales. Then for (a) it suffices to
apply this result to an online \lce martingale $N$ such that 
$M(\sigma\oplus \tau)=N(\tau\mid \sigma)$. In particular, let $(N_i)$ be a computable family of
online \lce martingales such that
$N(\tau\mid \sigma)=\sum_i N_i(\tau\mid \sigma)$.
Then $M(\sigma\oplus \tau)=\sum_i N'_i(\sigma\oplus \tau)$ where
$N'_i(\sigma\oplus \tau):=N_i(\tau\mid \sigma)$ are uniformly computable and even-parity.

A direct relativization of the argument in \citep[\S 6.3.1]{rodenisbook} 
for \lce supermartingales shows that it also holds for online supermartingales.
Hence for (b) it suffices to apply this fact to an online \lce supermartingale $N$ such that
$M(\sigma\oplus \tau)=N(\tau\mid \sigma)$.
The case of odd-parity martingales is symmetric. 
\end{proof}
Surprisingly, the closely related single-sided strategies can be more complex and do not necessarily
satisfy the decomposition of (i) of Lemma \ref{ttrAuakcfr}. The reason for this is that
the wager $\abs{M(\sigma)-M(\sigma\ast 0)}$ of a \lce  martingale is not necessarily \lce while
the wagers of an effective mixture of computable single-sided strategies are. In this case the closest analogue 
is that given $i<2$:
\begin{equation*}
\parbox{13cm}{for every strictly $i$-sided \lce martingale $M$, 
there exists a uniformly computable sequence $(N_i)$ of martingales such that 
the partial sums $S_n=\sum_{i<n} N_i$ are strictly $i$-sided  and converge to $M$. 
\citep[Lemma 2.7]{barmpalias2020monotonous}}
\end{equation*}
where {\em strictly $i$-sided} means that
$M(\sigma i)> M(\sigma \ast (1-i))$ for all $\sigma$.
The point here is that if we view $M$ as an enumeration of bets, we only know that the aggregate bet at any 
initial segment of outcomes is single-sided, while individual sub-bets in the enumeration could favor either outcome 0,1.
This issue was discussed by \citet[\S 2.2]{barmpalias2020monotonous} where these
sub-bets were called {\em intermediate}. For a decomposition into single-sided strategies we need the 
stronger hypothesis that $M$ is strongly l.c.e., which means that the wagers of $M$ are also l.c.e.:
\begin{equation*}
\parbox{13cm}{for every strongly \lce  strictly $i$-sided martingale $M$, 
there exists a uniformly computable sequence $(N_i)$ of $i$-sided martingales such that 
the partial sums $S_n=\sum_{i<n} N_i$ 
converge to $M$. \citep[Lemma 2.8]{barmpalias2020monotonous}}
\end{equation*}
This is the reason that
\citep[Theorem 3.3]{barmpalias2020monotonous}
was restricted to strongly \lce martingales: 
there exists $x$ with $\dim (x)=1/2$ and each single-sided strongly \lce martingale fails on $x$.
Constructing a non-random $x$ such that all single-sided \lce (super)martingales fail on $x$ is considerably
more challenging and requires different methods, see \citep{sided21}. Moreover we do not know
how low the effective Hausdorff dimension of such a real can be (other than it must be at least 1/2), and indeed
whether it can be less than 1.

An even-parity \lce supermartingale $M$ is {\em optimal} if for 
any even-parity \lce supermartingale we have $M'=\bigo{M}$;
a similar definition applies for odd-parity \lce supermartingales.
Let $\lambda$ denote the empty string.
\begin{lem}\label{WAcx8aVmQB}
There exist optimal even-parity and odd-parity \lce supermartingales $M,N$ respectively such that 
$M(\lambda)<1/2, N(\lambda)<1/2$.
\end{lem}\begin{proof}
The weighted sum of two even-parity \lce supermartingales is an even-parity \lce supermartingale.
Moreover there exists a universal enumeration $(M_i)$ of all \lce even-parity supermartingales $M$ with $M(\lambda)\leq 1$.
Hence $\sum_i 2^{-i}\cdot M_i$ is an optimal \lce even-parity supermartingale. A similar argument applies for the case
of odd-parity \lce supermartingales.
\end{proof}

\subsection{Speed of capital-gain and  fractal dimensions}\label{Sp9icqFzvI}
The martingale approach to algorithmic information theory was introduced by \citet{Schnorr:71, Schnorr:75}
who also showed some interest in the rate of success of (super)martingales $M$, and in particular the classes
\[
S_h(M)=\left\{x\ |\ \limsup_n \frac{M(x\restr_n)}{h(n)}=\infty\right\}
\hspace{0.3cm}\textrm{and}\hspace{0.3cm}
\hat S_h(M)=\left\{x\ |\ \liminf_n \frac{M(x\restr_n)}{h(n)}=\infty\right\}
\]
where $h:\Nat\to\Nat$ is a computable non-decreasing function.
 \citet{Lutz:00,Lutz:03} showed that the Hausdorff dimension of a class of reals can be characterized
by the exponential `success rates' of \lce supermartingales, and in that light defined the 
effective Hausdorff dimension $\dim (x) $ of a real $x$ as the infimum of the  $s\in (0,1)$ such that $x\in S_{h}(M)$ for some
\lce supermartingale $M$, where $h(n)=2^{(1-s)n}$.
 \citet{Mayordomo:02} showed that
\begin{equation}\label{kleTr2ZET}
\dim (x)=\liminf_n \frac{K(x\restr_n)}{n}
\end{equation}
where $K$ denotes the  \pf Kolmogorov complexity.
\ml random reals have effective dimension 1, 
but the converse does not hold. Moreover there
are computably random reals (no computable martingale succeeds on them) of effective dimension 0. 

\citet{Athreya02} showed that the {\em packing dimension} 
of \citet{Tricotpack, Sullivanpack}
can be effectivized and characterized in a similar way:  
\[
\Dim(x):=\inf\sqbrad{s\in (0,1)}{x\in \hat S_{h}(M)}=\limsup_n \frac{K(x\restr_n)}{n}
\]
where $h(n)=2^{(1-s)n}$ and the infimum is taken over all \lce supermartingales $M$.
We call $\Dim(x)$ the {\em effective packing dimension} of $x$.
Note that $\dim(x)\leq \Dim(x)$ for all $x$.
\begin{lem}\label{l33UQ7IjSb}
Suppose that $(V_i)$ is uniformly c.e., 
$V_i\subseteq 2^{2i}$ and for each $\sigma\in V_i$ at most three of the four 2-bit extensions
of $\sigma$ belong to $V_{i+1}$. Then each $x$ which has a prefix in each $V_i$ has $\Dim (x)\leq \log \sqrt 3$.
\end{lem}\begin{proof}
By the hypothesis, for each $\sigma\in V_i$ we have $\mu_{\sigma}(V_{i+1})\leq 3/4$.

Define $M$ by $M(\lambda)=1$ and for each $i$, $\sigma\in V_i$, $\sigma'\succ\sigma$ with $|\sigma'|=|\sigma|+2$ let
\[
M(\sigma')=
\left\{\begin{array}{cl}
(4/3)\cdot M(\sigma)&\textrm{if $\sigma'\in V_{i+1}$}\\
0&\textrm{if $\sigma'\not\in V_{i+1}$}
\end{array}\right\}.
\]
%\begin{itemize}
%\item $M(\sigma')=(4/3)\cdot M(\sigma)$ if $\sigma'\in V_{i+1}$
%\item $M(\sigma')=0$ if $\sigma'\not\in V_{i+1}$.
%\end{itemize}
% 
Clearly $M$ is \lce on the strings of even length where it is defined, and
\[
4M(\sigma)\geq M(\sigma\ast 00)+M(\sigma\ast 01)+M(\sigma\ast 10)+M(\sigma\ast 11)
\]
so it is uniquely extendible to a supermartingale on $\twomel$.
By the definition of $M$ it follows that 
for each $n$ and $\sigma\in 2^{2n}$ such that $\forall i\leq n\ \sigma\restr_{2i}\in V_i$
we have $M(\sigma)\geq (4/3)^n$ so
\[
\liminf_n \frac{M(x\restr_n)}{(4/3)^{n/2}}>0.
\]
For each $s>\log \sqrt 3$ we have $2^{(1-s)n}<(4/3)^{n/2}$ hence
\[
s>\log \sqrt 3
\hspace{0.3cm}\Rightarrow\hspace{0.3cm}
\liminf_n \frac{M(x\restr_n)}{2^{(1-s)n}}=\infty
\]
which shows that $\Dim (x)\geq \log \sqrt 3$, by the definition of effective packing dimension.
\end{proof}
For  further background on algorithmic dimension see \citep[Chapter 13]{rodenisbook}.

\subsection{The van Lambalgen theorem}
The case of single-parity betting is superficially relevant to the van Lambalgen theorem, 
so we briefly review the interesting ways that the latter fails under computability restrictions.

Recall that  $z$ is random with respect to the class $\MM$ of  supermartingales if
$M(z\restr_n)=\bigon$ for each $M\in\MM$. If $\MM$  contains somewhat effective supermartingales that
can be relativized to an oracle $y$, we may consider the class
$\MM^{y}$ of the members of $\MM$ relativized to oracle $y$ and say that
$z$ is $y$-random with respect to $\MM$ if it is random with respect to $\MM^{y}$.
We say that $x,y$ are
\begin{itemize}
\item  {\em mutually random} with respect to $\MM$ if
$x$ is $y$-random with respect to $\MM$ and $y$ is $x$-random with respect to $\MM$.
\item {\em weakly mutually random} with respect to $\MM$ if
$x$ is random with respect to $\MM$ and $y$ is $x$-random with respect to $\MM$; or
$x$ is $y$-random with respect to $\MM$ and $y$ is random with respect to $\MM$.
\end{itemize}
The van Lambalgen theorem says that the following are equivalent:
\begin{enumerate}[(i)]
\item  $x\oplus y$  is random with respect to the class $\MM$ of  \lce supermartingales;
\item $x,y$ are mutually random with respect to the class $\MM$ of  \lce supermartingales;
\item $x,y$ are weakly mutually random with respect to the class $\MM$ of  \lce supermartingales.
\end{enumerate}
Restriction to the computable strategies has the following effects on this equivalence.
We refer to randomness with respect to the class of computable supermartingales as {\em computable randomness};
mutual and weakly mutual computable randomness is defined analogously.
\begin{itemize}
\item there are $x,y$ such that $x\oplus y$ is computably random but $x,y$ are not mutually computably random;
hence (i)$\to$ (ii) fails for computable strategies; \citep{YuvanLamb}
\item if $x, y$ are mutually is computably random then $x\oplus y$ is computably random;  
hence (ii)$\to$ (i) holds for computable strategies; \citep{Bauwens20uvan} 
\item there are $x,y$ such that $x\oplus y$ is not computably random but $x,y$ are weakly mutually  computably random;
hence (iii)$\to$ (i) fails for computable strategies; \citep{Bauwens20uvan}.
\end{itemize}
\citet{miyaberute} studied van Lambalgen's theorem in the context of computable strategies and
a notion of {\em uniform relativization}, and \citet{Satyadevresourvanlan} considered the case of time-bounded
computable strategies.

Both the single-parity strategies we considered earlier and the relativized strategies in the 
van Lambalgen theorem use the information in the even bits of the outcome in order to bet on the odd bits (or vice-versa).
However the access to the even bits in the van Lambalgen theorem is unrestricted, while in the case of
single-parity strategies it is {\em online} in the sense that when placing a bet on round $t$ the player only has
access on the even bits $i<t$ (as well as the odd bits $<t$ where he was allowed to bet). 
This type of online advice was explored by \citet{ChernovSVV08} and \citet{stacsBauwens14},
and can be equivalently formulated in terms of {\em online program-size complexity} or {\em online semimeasures}.

\section{Integer-valued strategies and single-parity irreducibility}\label{9tswL9nmI}
When the wagers of a strategy are restricted to a specific set of values, their success may be limited.
Such restrictions were studied by
\citet{TeutChalcraft}, \citet{Peretzagainst}, \citet{Peretzwager}
while \citet{BienvenuST10}, \citet{Teu14agCWGnp}, \citet{JCSSbarmp15} focused on the specific case of integer wagers.
\citet{granuma} studied the case where the permissible granularity for the wagers decreases at given rates. 

We show that in the case of integer-valued supermartingales the reducibility to
single-sided or single-parity strategies no longer holds, even in the case of computable strategies.

\begin{defi}\label{gN4otBmdQj}
We define two computable 
integer-valued martingales $N, D$. 
Martingale $N$ only bets on outcome 1 and is given by:
\begin{itemize}
\item $N(\lambda)=5$ and if $N(\sigma)=0$ then $N(\sigma 0)=N(\sigma 1)=0$
\item otherwise,  $N(\sigma 0)=N(\sigma)-1$ and $N(\sigma\ast 1)=N(\sigma)+1$.
\end{itemize}
Martingale $D$ bets the minimal wager 1 on the two outcomes alternately and is given by:
\begin{itemize}
\item $D(\lambda)=5$ and if $D(\sigma)=0$ then $D(\sigma 0)=D(\sigma 1)=0$
\item if $D(\sigma)>0$ and $|\sigma|$ is even then $D(\sigma 0)=D(\sigma)+1$ and $D(\sigma 1)=D(\sigma)-1$ 
\item if $D(\sigma)>0$ and $|\sigma|$ is odd then $D(\sigma 0)=D(\sigma)-1$ and $D(\sigma 1)=D(\sigma)+1$. 
\end{itemize}
\end{defi}
\begin{lem}\label{pxbnVAGiWD}
If  $N(\rho)>2, D(\rho)>2$ then 
for each integer-valued supermartingale $M$
\begin{enumerate}[(i)]
\item if $M$ is single-parity,  
there exists $\tau\succ\rho$ such that $N(\tau)>c$ and  $\forall \tau' \succ\tau,\ M(\tau')=M(\tau)$.
\item if $M$ is single-sided  
there exists $\tau\succ\rho$ such that $D(\tau)>c$ and  $\forall \tau' \succ\tau,\ M(\tau')=M(\tau)$.
\end{enumerate}
\end{lem}\begin{proof}
For (i), without loss of generality, assume that $M$ is an even-parity supermartingale and for each odd position $i$ after $|\rho|$
let $\tau(i)=1$. Then for any choice of values in the even positions of $\tau$ after $|\rho|$ we have $N(\tau)\geq 1$.
Since $M(\tau\ast j)=M(\tau)$ whenever $|\tau|$ is even and since $M$ is integer-valued, we may fix a finite segment of the even positions
of $\tau$ after $|\rho|$ so that $\forall \tau' \succ\tau,\ M(\tau')=M(\tau)$. The conclusion then follows if we replace the constructed 
$\tau$ with $\tau\ast 1^c$.

For (ii), without loss of generality, assume that $M$ is  0-sided.
Recall that $D$ bets on 0 at even positions and on 1 at odd positions. 
If on each odd position $i$ after $|\rho|$ we 
let $\tau(i)=1$, then for any choice of values in the even positions of $\tau$ after $|\rho|$ we have $D(\tau)\geq 1$.
Since $M(\tau\ast 1)\leq M(\tau)$ for all $\tau$ and since $M$ is integer-valued, we may fix a finite segment of the even positions
of $\tau$ after $|\rho|$ so that $\forall \tau' \succ\tau,\ M(\tau')=M(\tau)$ and $|\tau|$ is even. 
The conclusion then follows if we replace the constructed 
$\tau$ with $\tau\ast (01)^c$.
\end{proof}
\begin{thm}[Integer-valued irreducibility]\label{5C9MB4KgaY}\ 
\begin{enumerate}[(a)]
\item There exists $z$ such that
some computable integer-valued martingale succeeds on $z$ but 
each computable integer-valued single-parity supermartingale fails on $z$.
\item There exists $z$ such that
some computable integer-valued martingale succeeds on $z$ but 
each computable integer-valued single-sided supermartingale fails on $z$.
\end{enumerate}
\end{thm}
\begin{proof}
Recall the martingales $N,D$ of Definition \ref{gN4otBmdQj}.
For (a), let $(M_i)$ be a list of all computable integer-valued single-parity supermartingales. 
By (i) of Lemma \ref{pxbnVAGiWD} we can define $z$ by initial segments so that
$M_i(z\restr_n)=\bigon$ for all $i$  and  $\lim_n N(z\restr_n)=\infty$. This proves (a).
For (b) let $(M'_i)$ be a list of all computable integer-valued single-sided supermartingales.
By (ii) of Lemma \ref{pxbnVAGiWD} we can define $z$ by initial segments so that
$M'_i(z\restr_n)=\bigon$ for all $i$  and  $\lim_n D(z\restr_n)=\infty$. This proves (b).
\end{proof}
Since  martingale $N$ only bets on 1, the definition of $z$ in Theorem \ref{5C9MB4KgaY}
by initial segments can be modified by interpolating arbitrarily long blocks `010101010\dots'
without losing the accumulated capital of $N$ along $z$. By choosing these segments sufficiently long 
we can ensure that the Kolmogorov complexity of $z$ drops sufficiently at the end of those segments, so that
\[
\liminf\frac{K(z\restr_n)}{n}=0
\]
so, by the characterizations of \S\ref{Sp9icqFzvI}, $\dim(z)=0$.

\section{Packing dimension and parity-betting: proof of Theorem \ref{Px6x4FnWQw}}\label{TCPLGlg6St}
We show that there are highly  predictable reals, where we can predict each consecutive pair
of bits in 3-out-of-4 guesses, yet 
no single-parity \lce supermartingale succeeds on them.
\begin{thm}\label{yV79J5ZGp1}
There exists $x$ and a \ml test $(V_i)$ such that $V_i\subseteq 2^{2i}$, $x\restr_{2i}\in V_i$, $\mu_{\sigma}(V_i)\leq 3/4$
for each $\sigma\in V_{i-1}$, and 
no single-parity \lce supermartingale succeeds on $x$.
\end{thm}
By Lemma \ref{l33UQ7IjSb} we have $\Dim(x)\leq \log \sqrt 3$, hence
Theorem \ref{Px6x4FnWQw}.
\begin{lem}\label{eyEw9DESd3}
Given non-negative $m_{00}, m_{10}$  
there exists an even-parity  martingale $M_0$ on $2^{\leq 2}$ such that
\begin{equation}\label{MV2jhnMXzk}
M_0(00)=m_{00}\fand M_0(10)=m_{10} \fand M_0(\lambda)=\max\{m_{00},m_{10}\}/2.
\end{equation}
Moreover for every
even-parity  martingale $M$ on $2^{\leq 2}$ such that
$M(00)\geq m_{00}\fand M(10)\geq m_{10}$ we have $M(\tau)\geq M_0(\tau), \tau\in 2^{\leq 2}$.
\end{lem}\begin{proof}
If $M$ is even-parity we have
\begin{equation}\label{PDnKT85wXr}
M(\lambda)=M(0)=M(1)=\frac{M(00)+M(01)}{2}=\frac{M(10)+M(11)}{2}
\end{equation}
which, along with \eqref{MV2jhnMXzk}, uniquely determines
the values of $M_0(01), M_0(11)$: 
\begin{itemize}
\item if $m_{00}\geq m_{10}$ then $M_0(11)= m_{00}- m_{10}$ and $M_0(01)=0$
\item if $m_{00}< m_{10}$ then $M_0(01)= m_{10}- m_{00}$ and $M_0(11)=0$
\end{itemize}
which shows the first clause of the claim. On the other hand, 
for any martingale $M$ with $M(00)\geq m_{00}, M(10)\geq m_{10}$,
condition \eqref{PDnKT85wXr} implies 
$M(\lambda)\geq\max\{m_{00},m_{10}\}$, which in turn implies 
$M(\tau)\geq M_0(\tau), \tau\in 2^{\leq 2}$ as required.
\end{proof}
\begin{lem}\label{KSltdSxMM7}
Given $m_{00}, m_{10}, n_0, n_1\geq 0$  
there exist unique even and odd betting martingales $M_0, N_0$ on $2^{\leq 2}$
such that $M_0(\lambda)=\max\{m_{00},m_{10}\}/2$, $M_0(00)=m_{00}, M_0(10)=m_{10}$ 
$N_0(0)=n_0, N_0(1)=n_1$ and
\begin{itemize}
\item for any even-parity martingale $M$ on $2^{\leq 2}$ with 
$M(00)\geq m_{00}\fand M(10)\geq m_{10}$
there exists an even-parity martingale $D_M$ on $2^{\leq 2}$ such that $M=M_0+D_M$.
\item for any odd-parity martingale $N$ on $2^{\leq 2}$ with 
$N(0)\geq n_{0}\fand N(1)\geq n_{1}$
there exists an odd-parity martingale $D_N$ on $2^{\leq 2}$ such that $N=N_0+D_N$.
\end{itemize}
\end{lem}\begin{proof}
Let $M_0$ be as in Lemma \ref{eyEw9DESd3} so $D_M:=M-M_0$ is a non-negative even-parity martingale.
Also let $N_0$ be the unique odd-parity martingale on $2^{\leq 2}$ determined by
$N_0(0)=n_0, N_0(1)=n_1$ so $D_N:=N-N_0$ is a non-negative odd-parity martingale.
\end{proof}
\begin{lem}\label{DUzvpITARE}
Given $m_{00}, m_{10}, n_0, n_1, c\geq 0$
suppose that $M,N$ are even and odd betting supermartingales such that 
$M(\eta 00)\geq m_{00}, N(\eta 0)\geq n_0, M(\eta 10)\geq m_{10}, N(\eta 1)\geq n_1$, 
$M(\eta)+N(\eta)\leq c$ and
\begin{equation}\label{yYok7QvsB6}
m_{00}+n_0\geq c
\fand
m_{10}+n_1\geq c.
\end{equation}
where $\eta$ is a string of even length.
Then 
\[
\parb{n_0\leq n_1 \hspace{0.1cm}\Rightarrow\hspace{0.1cm}M(\eta 01)+N(\eta 01)\leq c}
\fand 
\parb{n_0> n_1\hspace{0.1cm}\Rightarrow\hspace{0.1cm}M(\eta 11)+N(\eta 11)\leq c}.
\]
\end{lem}\begin{proof}
Without loss of generality we may assume $\eta=\lambda$, and that 
$M,N$ are  martingales: in the case that they are
mere supermartingales we may consider the unique
even and odd betting martingales  $M',N'$ that agree with $M, N$ on the strings $\eta\ast 2^{\leq 2}$;
then $M'(\lambda)\leq M(\lambda)$, $N'(\lambda)\leq N(\lambda)$, $N(i)=N(ij)=N'(ij)=N'(i), i,j<2$,
so the hypothesis on $M,N$ applies on $M',N'$
and the application of the statement on $M',N'$ gives
\begin{itemize}
\item  $n_0\leq n_1\hspace{0.2cm}\Rightarrow\hspace{0.2cm}M(01)+N(0)=M'(\eta 01)+N'(\eta 0)\leq c$ 
\item  $n_0> n_1\hspace{0.2cm}\Rightarrow\hspace{0.2cm}M(11)+N(1)=M'(11)+N'(1)\leq c$.
\end{itemize}
Given $m_{00}, m_{10}, n_0, n_1, M,N$ consider the martingales $M_0, N_0, D_M, D_N$
given by Lemma \ref{KSltdSxMM7} so
\begin{equation}\label{c93HW8RXbZ}
M+N=(M_0+N_0)+D_M+D_N
\hspace{0.3cm}\textrm{and}\hspace{0.3cm}
c_0:=M_0(\lambda)+N_0(\lambda)\leq c.
\end{equation}
There are two rounds of bets corresponding to the two bits:
\begin{itemize}
\item  $M_0$ bets only on the second round (his bets may be dependent on the outcome of the first round)
\item $N_0$ bets only on the first round, so $N_0(ij)=N_0(i), i,j<2$.
\end{itemize}
If $n_0\leq n_1$ then  $N_0(0)$ is the loser so 
by the first of \eqref{yYok7QvsB6}, $M_0(00)=m_{00}, N_0(0)=n_0$ and the last of \eqref{c93HW8RXbZ},
$M(\eta 00)$ has to win (at least) the amount $N_0(\lambda)-N_0(0)$ that $N$ loses under outcome 0, 
plus the difference $c-c_0$ from the initial capital $c_0$ of $M_0+N_0$.
As a consequence, $M_0(01)$ has to lose the same amount so, since $M(0)=M(\lambda)$:
\begin{equation}\label{IpWNBMUttt}
M_0(01)+N_0(01)=
M_0(01)+N_0(0)\leq M(0)-(c-c_0)+N(\lambda)=c_0-(c-c_0).
\end{equation}
By \eqref{c93HW8RXbZ} we have $D_M(\lambda)+D_N(\lambda)\leq c-c_0$.
Since  $D_M, D_N$ are single-parity martingales, at the end of the two rounds  and at any outcome 
they can at most double their initial capital, which is $\leq c-c_0$, so
\[
D_M(01)+ D_N(01)\leq 2\cdot (c-c_0).
\]
Combining the above with \eqref{IpWNBMUttt} and \eqref{c93HW8RXbZ} we get that $M(01)+N(01)$ is bounded above by:
\[
M_0(01)+N_0(01)+D_M(01)+ D_N(01)\leq c_0-(c-c_0) +2\cdot (c-c_0)=c_0+(c-c_0)=c
\]
as required. The case $n_0> n_1$ is symmetric.
\end{proof}
\begin{lem}\label{i8IYC2awRL}
Given $c\geq 0$, string $\eta$ of even length and \lce even and odd betting supermartingales 
$M,N$ we can effectively enumerate $V\subseteq 2^{\leq 2}$ such that $|V|\leq 3$ and if
$M(\eta)+N(\eta)\leq c$ then $\exists\tau\in V:\  M(\tau)+N(\tau)\leq c$.
\end{lem}\begin{proof}
We first enumerate $\eta00, \eta10$ into $V$ and wait until a stage $s$ such that
\[
M_s(\eta00)+N_s(\eta00)> c\fand M_s(\eta10)+N_s(\eta10)> c
\]
where $(M_t)$, $(N_t)$ are computable non-decreasing even and odd betting supermartingales converging to $M,N$.
These approximations can be effectively obtained from any \lce approximations to $M,N$, so the above is a
$\Sigma^0_1$ condition.
If and when this occurs at some stage $s$  we let $m_{00}:=M_s(\eta00)$,  $m_{10}:=M_s(\eta10)$, $n_0:= N_s(\eta00)$, $n_1:= N_s(\eta10)$ and
\begin{itemize}
\item if $n_0\leq n_1$ we enumerate  $\eta 01$ into $V$
\item if $n_0> n_1$ we enumerate  $\eta 11$  into $V$.
\end{itemize}
If the third string never gets enumerated into $V$, we have
$|V|= 2$ and $M(\tau)+N(\tau)\leq c$ for at least one of $00, 10$; otherwise
$|V|= 3$ and by Lemma \ref{DUzvpITARE} $M(\tau)+N(\tau)\leq c$ holds for the last string $\tau$ that was enumerated into $V$.
\end{proof}
We may now complete the proof of Theorem \ref{yV79J5ZGp1}.
Let $M, N$ be as in Lemma \ref{WAcx8aVmQB} and let
$D(\sigma):=M(\sigma)+N(\sigma)$, so $D(\lambda)\leq 1$.
Then by nested application of Lemma \ref{i8IYC2awRL}
we can define an array $(V_i)$  
satisfying the conditions of Lemma \ref{l33UQ7IjSb} and such that
for each $\sigma\in V_i$ such that $D(\sigma)\leq 1$ there exists $\sigma'\in V_{i+1}$ such that
$D(\sigma')\leq 1$. This condition shows that  there are infinite paths $x$ through the tree defined by $(V_i)$
such that $\forall n\ \parb{x\restr_{2n}\in V_n\fand D(x\restr_{2n})\leq 1}$.
By Lemma \ref{l33UQ7IjSb} it follows that 
$D(x\restr_n)=\bigon$
and by the optimality of $D$ by Lemma \ref{WAcx8aVmQB}, we have
$D'(x\restr_n)=\bigon$ for any single-parity \lce supermartingale $D'$.

\section{Hausdorff dimension and parity-betting: proof of Theorem \ref{OYNIGDVIET}}
Effective Hausdorff dimension dimension, as discussed in \S\ref{Sp9icqFzvI}, can  be
characterized of  in terms of effective statistical tests.
Given $s\in (0,1)$, an $s$-test is a uniformly \ce sequence $(V_i)$ of sets of strings such that
$\sum_{\sigma\in V_k} 2^{-s|\sigma|}<2^{-k}$ for each $k$. 
As reported  by \citet[\S 13.6]{rodenisbook}: 
\begin{equation}\label{IWvHgcPRHK}
\parbox{11cm}{given $s\in (0,1)$ one can effectively obtain an
effective list of all $s$-tests.}
\end{equation}
Since $s<1$, the condition $\sum_{\sigma\in V_k} 2^{-s|\sigma|}<2^{-k}$ means that
the length of each string in $V_k$ is more than $k$.
These observations will be used in the proof of Theorem \ref{OYNIGDVIET},
which is Lemma \ref{Fa8iXDhuyB} below.
We say that $x$ is {\em weakly $s$-random} if it {\em avoids all $s$-tests} $(V_i)$, 
in the sense that there are only finitely many 
$i$ such that $x$ has a prefix in $V_i$. 
\citet{MR1888278} showed that the weak $s$-randomness of $x$  is equivalent to   
$\exists c\ \forall n\ K(x\restr_n)>s\cdot n-c$, so  by \eqref{kleTr2ZET}:
\begin{equation}\label{DitiNX8zan}
\dim (x)=\sup \{s : \textrm{$x$ is weakly $s$-random}\}.
\end{equation}
Theorem \ref{OYNIGDVIET} is the following lemma.
\begin{lem}\label{Fa8iXDhuyB}
If $\dim z<1/2$ then there are strongly \lce martingales $N,T,$ such that
$N$ is even-parity, $T$ is odd-parity  and $\lim_n N(z\restr_n)=\lim_n T(z\restr_n)=\infty$.
\end{lem}\begin{proof}
Let $(V_i)$ be a universal 1/2-test, so that every $z$ which is not weakly 1/2-random has prefixes in infinitely many $V_i$.
We define computable families $(N_{\sigma}), (T_{\sigma})$ of even-parity and odd-parity strategies respectively, indexed by strings, and let:
\[
N:=\sum_i \sum_{\sigma\in V_i} N_{\sigma}
\hspace{0.5cm}\textrm{and}\hspace{0.5cm}
T:=\sum_i \sum_{\sigma\in V_i} T_{\sigma}.
\]
Since each $N_{\sigma}$ is even-parity,  the same is true of $N$, and in the same way $T$ is odd-parity.
For each $i,\sigma$ strategy $N_{\sigma}$ starts with $N_{\sigma}(\lambda)=2^{-|\sigma|/2}$ and 
at each $\rho\prec\sigma$ of even length it bets all capital on $\sigma(|\rho|)$, while  placing no bets at odd positions. Formally:
\[
N_{\sigma}(\rho)=
\left\{\begin{array}{cl}
0&\textrm{if $|\hat\rho|$ is even and $\rho\not\preceq\sigma$}\vspace{0.1cm}\\
2\cdot N_{\sigma}(\hat\rho)&\textrm{if $|\hat\rho|$ is even and $\rho\preceq\sigma$}\vspace{0.1cm}\\
N_{\sigma}(\hat\rho)&\textrm{if $|\hat\rho|$ is odd or $|\rho|>|\sigma|$}
\end{array}\right\}
\hspace{0.3cm}\textrm{where $\hat\rho$ denotes the predecessor of $\rho$.}
\]
Then for each $\rho\succ\sigma$ we have $N_{\sigma}(\rho)=1$.
If $\dim z<1/2$ then $z$ has prefixes in infinitely many $V_i$ so $\lim_n N(z\restr_n)=\infty$.
The definition of the odd-parity $T_{\sigma}$ is analogous, as well as the proof that
$\lim_n N(z\restr_n)=\infty$, provided that $\dim z<1/2$.
\end{proof}
\section{Conclusion}\label{mcSrWpox}
Muchnik's paradox says that some reals are predictable with respect to \lce supermartingales
but unpredictable with respect to single-parity \lce supermartingales. Informally,
{\em some \lce strategies are irreducible to single-parity strategies}. We showed that reals with $\dim(x)<1/2$
are predictable with respect to both odd and even parity \lce martingales. We also showed that
there are reals $x$ with effective packing dimension as low as $\log \sqrt{3}\approx 0.79$, yet 
no single-parity \lce supermartingale succeeds on $x$. The following question arises:
\begin{equation*}
\parbox{13.5cm}{how low can the effective Hausdorff and packing dimensions of reals exhibiting Muchnik's paradox be?}
\end{equation*}
By Theorem \ref{OYNIGDVIET}, any such real has effective Hausdorff dimension at least $1/2$,
while Theorem \ref{Px6x4FnWQw} gives an example with effective packing dimension at most $\log\sqrt{3}$.
These bounds do not determine the least possible effective Hausdorff or packing dimension.
Questions regarding the predictability of reals exhibiting Muchnik's paradox
are questions about the power of single-parity betting.

We also exhibited Muchnik's paradox in the case of computable integer-valued  (super)martingales.
This is interesting since this phenomenon does not occur in the case of computable (super)martingales.
Integer-valued strategies represent a specific example of a wager-restriction, so
a related question is 
\begin{equation*}
\parbox{12cm}{which wager restrictions permit the irreducibility of a computable strategy
to computable single-parity strategies?}
\end{equation*}
Finally we discussed the case of \lce single-sided strategies and why they appear to be more powerful then
the single-parity \lce strategies. Using different methods, \citet{sided21}  
showed that there exists non-random $x$ such that no single-sided 
\lce supermartingale succeeds on $x$. The following question remains open:
\begin{equation*}
\parbox{13cm}{
How small can the effective Hausdorff and packing dimensions of $x$ be
if no single-sided \lce supermartingale succeeds on $x$?}
\end{equation*}
Only the lower bound 1/2 is known; for example, we don't know if such $x$ can have $\dim(x)<1$.

% Bibliography included directly so that no external .bib or .bbl is needed.

\end{document}